\documentclass[12pt]{amsart}

\usepackage{amsthm}
\usepackage{enumerate}
\usepackage{amssymb}
\usepackage{hyperref}

\newtheorem{thm}{Theorem}[section]
\newtheorem{defn}[thm]{Definition}
\newtheorem{lem}[thm]{Lemma}
\newtheorem{prop}[thm]{Proposition}
\newtheorem{cor}[thm]{Corollary}
\newtheorem{rem}[thm]{Remark}
\newtheorem*{nota}{Notation}
\newtheorem*{thmd}{Theorem}

\title[Spectral synthesis in multidimensional Fourier algebras]{Spectral synthesis in multidimensional Fourier algebras}

\author{Kanupriya}
\address{Kanupriya,\newline\indent Department of Mathematics,\newline\indent Indian Institute of Technology Delhi,\newline\indent New Delhi - 110016, India.}
\email{kanupriyawadhawan3@gmail.com}
\author{N. Shravan Kumar}
\address{N. Shravan Kumar,\newline\indent Department of Mathematics,\newline\indent Indian Institute of Technology Delhi,\newline\indent New Delhi - 110016, India.}
\email{shravankumar.nageswaran@gmail.com}

\begin{document}
	
	
	
\begin{abstract}
    Let $G$ be a locally compact group and let $A^n(G)$ denote the $n$-dimensional Fourier algebra, introduced by Todorov and Turowska. We investigate spectral synthesis properties of the multidimensional Fourier algebra $A^n(G).$ In particular, we prove versions of the subgroup lemma, injection, and inverse projection theorems for both spectral sets and Ditkin sets. Additionally, we provide a result on the parallel synthesis between $A^n(G)$ and $A^{n+1}(G)$ and finally prove Malliavin's theorem.
\end{abstract}
	
\maketitle
	
\section{Introduction}
	
    The Fourier algebra $A(G)$ of a (non-abelian) locally compact group $G$ has been extensively studied since its appearance, in 1964, in the pioneering work of Eymard \cite{Eym1}. This algebra is important because it is a commutative, regular, and semisimple Banach algebra with its carrier space homeomorphic to $G.$ Also, if $G$ is abelian, then $A(G),$ via the Fourier transform, is isometrically isomorphic to $L^1(\widehat{G}).$ Here $\widehat{G}$ denotes the Pontrjagin dual of $G.$
	
    Studying the properties of the group algebra of a locally compact abelian group is one of the classical topics of abstract harmonic analysis. Among all, the study of spectral sets is very notable. Some landmarks in this context include the injection theorem, the projection theorem, and Malliavin's theorem.
	
    In 2010, Todorov and Turowska \cite{ToTu} introduced the multidimensional analogue of the Fourier algebra. This was motivated by the fact that, when $G$ is abelian, the two-dimensional Fourier-Stieltjes algebra, via the Fourier transform, coincides with the space of bimeasures on $\widehat{G}\times\widehat{G}$ \cite{GrSc}. This paper continues to study the spectral synthesis properties for the multidimensional Fourier algebra.
	
	If $H$ is a closed subgroup of a locally compact group $G$ and $E\subseteq H$ is closed, then $E$ is a set of synthesis for $A(G)$ if and only if $E$ is a set of synthesis for $A(H).$ This is known as the injection theorem, a classical result in the abelian case and is due to Reiter \cite{Rei}. The nonabelian version is due to Kaniuth and Lau \cite{KL} and independently by Parthasarathy and Prakash \cite{PaPr1}. For other extensions see \cite{Loh, Der1, PaPr2, PaNSK}. In Corollary \ref{Inj_Thm_Sp_Syn} of Section 3, we prove the injection theorem for the spectral synthesis concerning the multidimensional Fourier algebra. More precisely, we prove the following.
    \begin{thmd}
        Suppose that $H$ is a closed subgroup of a locally compact group $G$ and $X$ an $A^n(G)$-submodule of $VN^n(G).$ Then a closed subset $E$ of $H^n$ is a set of $X_H$-synthesis for $A^n(H)$ if and only if $E$ is a set of $X$-synthesis for $A^n(G).$
    \end{thmd}
	
	One of the main ingredients in the injection theorem is the subgroup lemma, which states that a closed subgroup of a locally compact group $G$ is a set of spectral synthesis for $A(G).$ This result is again due to Reiter in the abelian case, while the nonabelian version is due to Takesaki and Tatsuuma \cite{TaTa2}. Theorem \ref{SubGp_Lemma_Spectral} of this paper gives a multidimensional version of the subgroup lemma. We show that if $H$ is a closed subgroup of $G,$ then $H^n$ is a set of spectral synthesis for $A^n(G).$
	
	The analogue of the above injection theorem and subgroup lemma can be given for the Ditkin sets. Again, in the abelian case this is a classical result due to Reiter \cite{Rei}. The subgroup lemma for Ditkin sets in the nonabelian case is due to Derighetti \cite{Der2}. The injection theorem is due to Ludwig and Turowska \cite{LuTu} and independently by Derighetti \cite{Der3}. Theorem \ref{SubGp_Lemma_Ditkin} and Theorem \ref{Inj_Thm_Ditkin_1} provide the subgroup lemma and the injection theorem for local Ditkin sets, respectively, where the subgroups are normal. The precise statement is as follows. 
    \begin{thmd}
        Let $H$ be a closed normal subgroup of $G.$ 
        \begin{enumerate}
            \item Then $H^n$ is a locally Ditkin set for $A^n(H).$ 
            \item If $E$ is a closed subset of $H^n,$ then $E$ is a local Ditkin set for $A^n(G)$ if and only if $E$ is a local Ditkin set for $A^n(H).$
        \end{enumerate}
    \end{thmd}
    
    Further, in Theorem \ref{Inj_Thm_Ditkin_2}, we prove the injection theorem without the assumption of normality on subgroups but here the group is amenable. 
	
	Yet another classical result of Reiter \cite{Rei}, the inverse projection theorem for sets of synthesis, states that if $H$ is a closed subgroup of the abelian group $G,$ then $\widetilde{E}\subseteq G/H$ is a set of synthesis for $A(G/H)$ if and only if its pullback is a set of synthesis for $A(G).$ Derighetti \cite{Der1} and Lohou\'{e} \cite{Loh} independently proved the inverse projection for the nonabelian case. For other extensions see \cite{PaNSK}. A partial, nonabelian version for local Ditkin sets is due to Derighetti \cite{Der1} and the reader is referred to \cite{PaNSK2} for a complete version. Theorem \ref{IPT} of the present paper contains an inverse projection theorem for spectral synthesis and strong Ditkin sets in the case of compact groups. The exact statement is here.
    \begin{thmd}
        Let $H$ be a closed normal subgroup of a compact group $G$ and let $\widetilde{E}$ be a closed subset of $(G/H)^n.$ Then 
        \begin{enumerate}
            \item $\widetilde{E}$ is a set of synthesis for $A^n(G/H)$ if and only if $q_n^{-1}(\widetilde{E})$ is a set of spectral synthesis for $A^n(G).$
            \item $\widetilde{E}$ is a strong Ditkin set for $A^n(G/H)$ if and only if $q_n^{-1}(\widetilde{E})$ is a strong Ditkin set for $A^n(G).$
        \end{enumerate}
    \end{thmd}
	
    In Section 5, we prove a result of the parallel synthesis between $A^n(G)$ and $A^{n+1}(G).$ This result is for both sets of spectral synthesis and strong Ditkin sets in compact groups. The precise statements are as follows.
    \begin{thmd}
		Assume that $G$ is compact. Let $E \subseteq G^n$ be closed and let $$E^* = \{(x_1,x_2,...x_{n+1}) :(x_1x_2^{-1}, x_2x_3^{-1},..., x_nx_{n+1}^{-1}) \in E\}.$$
		(i) The set $E$ is a set of spectral synthesis for $A^n(G)$ if and only if $E^*$ is a set of spectral synthesis for $A^{n+1}(G)$.\\
		(ii) The set $E$ is a strong Ditkin set for $A^n(G)$ if and only if $E^*$ is a strong Ditkin set for $A^{n+1}(G)$.
	\end{thmd}
	A similar result can be found in \cite{AlToTu}.
	
	One of the celebrated results of Abelian harmonic analysis is Malliavin's theorem on the failure of spectral synthesis. In fact, for a locally compact abelian group $G,$ Malliavin showed that every closed subset of $G$ is a set of spectral synthesis if and only if $G$ is discrete \cite{Ma1, Ma2, Ma3} (see also \cite{Ru}). The nonabelian version of this result is due to Kaniuth and Lau \cite{KL}. For other generalizations see \cite{Kan1, PaPr2}. In the last part of the present paper, we prove the multidimensional analogue of Malliavin's theorem.

\section{Notations and Preliminaries}

    \subsection{Tensor product of operator spaces}
        In this section, we go through some fundamentals of operator space. As tensor products will be crucial, our main goal in creating this section is to compile the necessary background information on these subjects.

        Let $X$ be a linear space and $M_n(X)$ denotes the space of all $n\times n$ matrices with entries from the space $X.$ An {\it operator space} is a complex vector space $X$ together with an assignment of a norm $\|\cdot\|_n$ on the matrix space $M_n(X),$ for each $n\in\mathbb{N},$ such that
        \begin{enumerate}[(i)]
            \item $\|x\oplus y\|_{m+n}=max\{\|x\|_m,\|y\|_n\}$ and
            \item $\|\alpha x\beta\|_n\leq\|\alpha\|\|x\|_m\|\beta\|$
        \end{enumerate}
        for all $x\in M_m(X),$ $y\in M_n(X),$ $\alpha\in M_{n,m}$ and $\beta\in M_{m,n}.$ Let $\varphi:X\rightarrow Y$ be a linear transformation from operator space $X$ to operator space $Y$ and $\varphi_n$ denotes the {\it$n^{\mbox{th}}$-amplification} of $\varphi$ which is a linear transformation from $M_n(X)$ to $M_n(Y)$ given by $\varphi_n([x_{ij}]) :=[\varphi(x_{ij})].$ This linear transformation $\varphi$  is said to be {\it completely bounded} if $\sup\{\|\varphi_n\|:n\in\mathbb{N}\}<\infty.$ Let $\mathcal{CB}(X,Y)$ denotes the space of all completely bounded linear mappings from $X$ to $Y$ equipped with the norm, denoted $\|\cdot\|_{cb},$ $$\|\varphi\|_{cb}:=\sup\{\|\varphi_n \|:n\in\mathbb{N}\},\ \varphi\in\mathcal{CB}(X,Y).$$ The map $\varphi$ is a said to be a {\it complete isometry (complete contraction)} if $\varphi_n$ is an isometry (a contraction) for each $n\in\mathbb{N}.$

        Let $X$ and $Y$ be operator spaces,  the {\it Haagerup tensor norm} of $u\in M_n(X\otimes Y)$ is given by $$\|u\|_h=\inf\left\{ \|x\|\|y\|:u=x\odot y,x\in M_{n,r}(X), y\in M_{r,n}(Y),r\in\mathbb{N} \right\}.$$ Here $x \odot y$ denotes the inner matrix product of $x$ and $y$ defined as $(x \odot y)_{i,j} = \underset{k=1}{\overset{r}{\sum}} x_{i,k} \otimes y_{k,j},$ where $x \odot y \in M_{n}(X \otimes Y).$ See \cite[Chapter 9]{ER3}. Now, the space $X\otimes^h Y$ denotes the resulting operator space corresponding to the Haagerup tensor norm. In the case when $X$ and $Y$ are C*-algebras, for $n=1,$ the Haagerup norm can be written as follows. For $u\in X\otimes^h Y,$ $$\|u\|_h=\inf \left\{ \left\| \underset{n\in\mathbb{N}}{\sum} x_n x_n^\ast \right\|_X^{1/2} \left\| \underset{n\in\mathbb{N}}{\sum} y_n^\ast y_n \right\|_Y^{1/2} :u= \underset{n\in\mathbb{N}} {\sum}x_n\otimes y_n \right\}.$$ 
            
        The {\it extended Haagerup tensor product} of $X$ and $Y,$ denoted $X\otimes^{eh} Y$ is defined as the space of all normal multiplicatively bounded functionals on $X^\ast\times Y^\ast.$ By \cite{ER2}, $(X\otimes^h Y)^\ast$ and $X^\ast\otimes^{eh}Y^\ast$ are completely isometric. Given dual operator spaces $X^\ast$ and $Y^\ast,$ the {\it $\sigma$-Haagerup tensor product (or normal Haagerup tensor product)} is defined by $$X^\ast\otimes^{\sigma h}Y^\ast=(X\otimes^{eh}Y)^\ast.$$ Further, the following inclusions hold completely isometrically: $$X^\ast\otimes^h Y^\ast \hookrightarrow X^\ast\otimes^{eh} Y^\ast \hookrightarrow X^\ast\otimes^{\sigma h} Y^\ast.$$
        \begin{nota}
            The n-tensor product of a space X is denoted by $$\otimes_{n}X=\underbrace{ X \otimes X \otimes\cdots\otimes X}_n.$$ Now, we define $\otimes_{n}^h X,$ $\otimes_{n}^{eh} X,$ and $\otimes_{n}^{\sigma h} X$ as the completion of $\otimes_{n}X$ with respect to their corresponding norms.
        \end{nota}
        We refer to \cite{ER1, ER2} to know more about these tensor products. For further details on operator spaces, the reader can refer to \cite{ER3}.

    \subsection{Multidimensional Fourier algebra}
        The multidimensional Fourier algebra, introduced by Todorov and Turowska, is defined in this section after the introduction of the fundamental Fourier algebra. Let $G$ be a locally compact group and fix a left Haar measure $dx$ on $G$. Let $\lambda_G$ denotes the left regular representation of $G$ acting on Hilbert space $L^2(G)$ by left translates: $\lambda_G(s)f(t) = f(s^{-1}t)$ for some $s,t \in G$ and $f \in L^2(G).$ The group von Neumann algebra $VN(G)$ is smallest self adjoint subalgebra in $\mathcal{B}(L^2(G)$ that contains all $\lambda_G(s)$ for all $s \in G$ and is closed in weak operator topology. The Fourier algebra $A(G)$ is actually the predual of $VN(G)$ and every function $u$ in $A(G)$ is represented by $u(x) = \langle \lambda_G(x)f,g \rangle$ with $\|u\|_{A(G)} = \|f\|_2\|g\|_2.$ The duality between $A(G)$ and $VN(G)$ is given by $\langle T,u \rangle = \langle Tf,g \rangle.$ For more details refer Eymard \cite{Eym1}. We shall denote by $G^n$ the group $G\times G\times\cdots\times G\ (n\mbox{-times}).$

        The multidimensional version of Fourier algebra denoted by $A^n(G)$, is actually defined as collection of all functions $f\in L^{\infty}(G^n)$ such that a normal completely bounded multilinear functions $\Phi$ on $\underbrace{VN(G) \times VN(G) \times \cdots \times VN(G)}_n$ satisfying $f(x_1,x_2,...,x_n)= \Phi(\lambda(x_1),...,\lambda(x_n))$ From \cite{ToTu}, $A^n(G)$ coincides with $\otimes_n^{eh} A(G)$. Also, the dual of $A^n(G)$ is completely isometrically isomorphic to $VN^n(G):= \otimes_n^{\sigma h}VN(G)$. See \cite{ToTu} for more details.
	
\section{Functorial properties and Support}
    In this section, we prove some functorial properties of $A^n(G).$ We also study the properties of the support of an element of $VN^n(G).$
	
    \subsection{Functorial Properties} 
        We shall begin this section by proving certain functorial properties of $A^n(G)$ in the spirit of \cite{Wo}. Most of these properties are direct consequences of \cite{ER2, Her, Wo}.
        \begin{nota}
            Given a closed subgroup $H$ of a locally compact group $G,$ Let $r_n:A^n(G)\rightarrow A^n(H)$ denote the restriction map, i.e., $$r_n(u_1\otimes u_2\otimes\cdots\otimes u_n)=(u_1\otimes u_2\otimes\cdots\otimes u_n)|_{H^n}=u_1|_H\otimes u_2|_H\otimes\cdots \otimes u_n|_H.$$
        \end{nota}
	  \begin{lem}\label{FunctProp}
            Let $H$ be a closed subgroup of a locally compact group $G.$ For $u\in A^n(H),$ let $u^\circ$ denote the function on $G^n$ that is $u$ on $H^n$ and vanishes outside $H^n.$
	    \begin{enumerate}[(i)]
	           \item The restriction mapping $r_n:A^n(G)\rightarrow A^n(H)$ is a complete contraction.
		      \item If $H$ is open, then $A^n(H)$ is completely isometrically isomorphic to                 $(A^n(G))^\circ$ $:=\{u^\circ:u\in A^n(H)\}.$
		      \item Given $u\in A^n(H),$ there exists $v\in A^n(G)$ such that $\|u\|_{A^n(H)}=\|v\|_{A^n(G)}$ and $v|_{H^n}=u.$
		      \item The spaces $A^n(H)$ and $A^n(G)/I_{A^n(G)}(H^n)$ are completely isometric to each other.
		      \item If $H$ is normal, then the inclusion of $A^n(G/H)$ inside $A^n(G)$ is a complete isometry.
	           \item The inclusion of $A(G^n)$ inside $A^n(G)$ is a complete contraction with dense range.
		  \end{enumerate}
	   \end{lem}
	   \begin{proof}
	       $(i)$ By \cite[Proposition 4.3]{Wo} the restriction map on the Fourier algebra $A(G)$ is a complete contraction. Hence by \cite[Pg. 144]{ER2}, the map $r_n$ is a complete contraction.
		
	       $(ii)$ This follows from \cite[Proposition 4.2]{Wo} and \cite[Lemma 5.4]{ER2}.
		
	       $(iii)$ This is a consequence of \cite[Theorem 1b]{Her} and \cite[Lemma 5.4]{ER2}.
		
	       $(iv)$ Let $Q^n:A^n(G)\rightarrow A^n(G)/I_{A^n(G)}(H^n)$ denote the natural quotient map. Now consider the map $\Gamma:A^n(G)/I_{A^n(G)}(H^n)\rightarrow A^n(H)$ given by $\Gamma(Q^n(u))=u|_{H^n}.$ It is clear that $\Gamma$ is well-defined. Further, it can also be seen that $\Gamma$ is norm decreasing. Now, given any $v \in A^n(H)$, by $(iii)$ there exists $u \in A^n(G)$ such that $u|_{H^n} = v$ and $||u||_{A^n(G)} = ||v||_{A^n(G)}$. Thus, $\Gamma$ is the required isometric isomorphism.
		
	       We now claim that $\Gamma$ is a complete isometry. Observe that $(I_{A^n(G)} (H^n))^{\perp} = VN_H(G).$ Further, by \cite[Lemma 3.1]{KL}, $VN_H^n(G)$ and $VN^n(H)$ are completely isometric to each other. Thus $A^n(G)/I_{A^n(G)} (H^n)$ and $A^n(H)$ are completely isometric to each other.
		
	       $(v)$ This follows from the fact that the inclusion of $A(G/H)$ inside $A(G)$ is a complete isometry and the fact that the extended Haagerup tensor product is injective.
		
	       $(vi)$ This follows from Theorem 7.1.1 and Theorem 7.2.4 of \cite{ER2} and the extended Haagerup norm is a matrix subcross norm.
	   \end{proof}
	
	\subsection{Support of element of $VN^n(G)$}
	   We now define and study the support of an element of the dual $VN^n(G),$ of $A^n(G).$ The results of this subsection are motivated from \cite[Section 2.5]{KL2}. As most of the results follow as in the case $n=1,$ we omit the proofs.
	    \begin{prop}\label{Supp_Equi_Defn}
	       Let $T\in VN^n(G)$ and $a_1,a_2,\ldots,a_n\in G.$ Then the following are equivalent:
		  \begin{enumerate}[(i)]
		      \item The operator $\lambda_G(a_1)\otimes\lambda_G(a_2) \otimes\cdots \otimes\lambda_G(a_n)$ is the weak*-limit in $VN^n(G)$ of operators of the form $u\cdot T$ where $u\in A^n(G).$
		      \item For every neighbourhood $U$ of $(a_1,a_2,\ldots,a_n)\in G^n,$ $\exists\ u\in A^n(G)$ such that $supp(u)\subset U$ and $\langle u,T \rangle\neq 0.$
        		\item If $u\in A^n(G)$ is such that $u\cdot T=0$ then $\langle u,\lambda_G(a_1)\otimes\lambda_G(a_2)\otimes\cdots\otimes\lambda_G(a_n) \rangle =0.$
	       \end{enumerate}
	   \end{prop}
	   \begin{defn}
		  For $T\in VN^n(G),$ the support of $T,$ denoted $supp(T),$ is the set of all $(a_1,a_2,\ldots,a_n)\in G^n$ satisfying any of the three equivalent conditions of the above proposition.
	   \end{defn}
	   \begin{lem}\label{Supp_prop1}
		  Let $G$ be a locally compact group and $T\in VN^n(G).$ 
		  \begin{enumerate}[(i)]
		      \item If $T\neq 0,$ then $supp(T)$ is a nonempty closed subset of $G^n.$
		      \item For $u\in A^n(G),$ $supp(u\cdot T)\subseteq supp(u)\cap supp(T).$
		      \item $supp(T)$ is the smallest closed set among all closed subsets $C$ of $G^n$ with the property that if $u\in A^n(G)$ has compact support and vanishes in an open $U$ containing $C,$ then $\langle u,T \rangle=0.$
		      \item $supp(T)$ is the smallest among all closed subsets $C$ of $G^n$ with the property that given any neighbourhood $U$ containing $C$ such that $G^n\setminus C$ is relatively compact, $T$ is the w*-limit of finite linear combinations of $\underset{1}{\overset{n}{\otimes}}\ \lambda_G(a_i),$ $(a_1,a_2,\ldots,a_n)\in U.$
		      \item Suppose that $(T_\alpha)_{\alpha\in\wedge}$ is a net in $VN^n(G)$ converging to $T$ in the w*-topology such that $supp(T_\alpha)\subseteq F$ for all $\alpha\in \wedge,$ for some closed set $F\subseteq G^n,$ then $supp(T)\subseteq F.$ 
		      \item If $S\in VN^n(G),$ then $supp(T+S)\subseteq supp(T)\cup supp(S).$ Further, equality holds if $supp(T)\cap supp(S)=\emptyset.$
		  \end{enumerate}
    	\end{lem}
	    \begin{lem}\label{Supp_prop2}
		  \mbox{}
		  \begin{enumerate}[(i)]
		      \item If $T_1,T_2,\ldots,T_n\in VN(G),$ then $supp\left(\underset{1}{\overset{n}{\otimes}}T_i\right)\subseteq \underset{1}{\overset{n}{\Pi}}supp(T_i).$
		      \item If $\alpha\in \mathbb{C}$ and $T\in VN^n(G),$ then $supp(\alpha T)=supp(T).$
		  \end{enumerate}
	   \end{lem}
	   \begin{proof}
		  $(i)$ We shall prove this for $n=2$ and the general case follows along the same lines. Let $(x_1,x_2) \in supp(T_1 \otimes T_2)$. Let $U_1, U_2$ be neighbourhoods of $x_1$ and $x_2$ respectively. For $V= U_1 \times U_2,$ there exists $u\in A^2(G)$ such that $\langle u, T_1 \otimes T_2 \rangle \neq 0.$ So, there exists a sequence $\left\{ \sum_{1}^{m_n} u_1^{i,n} \otimes u_2^{i,n} \right\} \subseteq A(G) \otimes A(G)$ such that 
		  \begin{center}
		      $u= \displaystyle\lim_{n \to \infty} \sum_{1}^{m_n} u_1^{i,n} \otimes u_2^{i,n}$ and  $\displaystyle\lim_{n \to \infty} \sum_{1}^{m_n} T_1(u_1^{i,n}) T_2(u_2^{i,n}) \neq 0.$ 
	       \end{center}
	       This implies that there exists $u_1, u_2 \in A(G)$ such that $\langle u,T_1 \rangle \neq 0, \langle u_2, T_2 \rangle \neq 0$, $supp(u_1) \subseteq U_1$ and $supp(u_2) \subseteq U_2.$ Therefore, $x_1 \in supp(T_1)$ and $x_2 \in supp(T_2).$ Hence the proof. 
		
	       $(ii)$ As above, we will prove this only for $n=2.$ Note that, without loss of generality, we can assume that $T$ is of the form $T_1\otimes T_2,$ where $T_1, T_2\in VN(G).$ Then, for any $\alpha\in\mathbb{C},$ $$supp(\alpha(T_1\otimes T_2))= supp(\alpha T_1)\times supp(T_2) = supp(T_1)\times supp(\alpha T_2).$$ Hence proved.
	    \end{proof}
	
\section{Subgroup lemma and injection theorem for sets of spectral synthesis}
	In this section, we aim to prove the multidimensional analogue of the subgroup lemma and the injection theorem for sets of spectral synthesis. 
	
	\subsection{Spectral synthesis in commutative Banach algebras}
	   We shall begin by defining the notions of spectral sets and Ditkin sets that will be used throughout this paper.
	
	   Let $\mathcal{A}$ be a regular, semisimple, commutative Banach algebra with the Gelfand structure space $\Delta(\mathcal{A}).$ For a closed ideal $I$ of $\mathcal{A},$ the zero set of $I,$ denoted $Z(I),$ is a closed subset of $\Delta(\mathcal{A})$ defined as $$Z(I)=\{a\in \Delta(\mathcal{A}):\widehat{a}(x)=0\ \forall\ x\in I\}.$$ For a closed subset $E\subset\Delta(\mathcal{A}),$ we define the following ideals in $\mathcal{A}:$
	   \begin{eqnarray*}
	       j_{\mathcal{A}}(E) &=& \{a\in\mathcal{A}:\widehat{a}\mbox{ has compact support disjoint from E} \}\\ 
	       J_{\mathcal{A}}(E) &=& \overline{j_{\mathcal{A}}(E)} \\
	       I_{\mathcal{A}}(E) &=& \{a\in\mathcal{A}: \widehat{a} = 0\mbox{ on }E\}.
        \end{eqnarray*}
        Note that $J_{\mathcal{A}}(E)$ and $I_{\mathcal{A}}(E)$ are closed ideals in $\mathcal{A}$ with the zero set equal to $E$ and $j_{\mathcal{A}}(E)\subseteq I \subseteq I_{\mathcal{A}}(E)$ for any ideal $I$ with zero set $E.$ $E$ is said to be a {\it set of spectral synthesis} (or a \textit{spectral set}) for $\mathcal{A}$ if $I_{\mathcal{A}}(E) = J_{\mathcal{A}}(E).$ 
	
    	A closed $E$ is a \textit{Ditkin set} if for every $a\in I_{\mathcal{A}}(E),$ there exists a sequence $\{a_n\}\subset j_{\mathcal{A}}(E)$ such that $a.a_n$ converges in norm to $a.$ If the sequence can be chosen in such a way that it is bounded and is the same for all $a\in I_{\mathcal{A}}(E),$ then we say that $E$ is a \textit{strong Ditkin set.} Note that every Ditkin set is a set of spectral synthesis. 
	
    	For more on spectral synthesis see \cite{Kan1, Rei}.
	
    	Let $X$ be a $\mathcal{A}$-submodule of $\mathcal{A}^\ast.$ For a closed subset $E\subset\Delta(\mathcal{A}),$ we define the following ideals in $\mathcal{A}:$
    	\begin{eqnarray*}
	       I_{\mathcal{A}}^X(E) &=& \{a\in\mathcal{A}:\langle a,\varphi \rangle = 0\ \forall\ \varphi\in X\cap I_{\mathcal{A}}(E)\} \\
	       J_{\mathcal{A}}^X(E) &=& \{a\in\mathcal{A}:\langle a,\varphi \rangle = 0\ \forall\ \varphi\in X\cap J_{\mathcal{A}}(E)\}.
        \end{eqnarray*}
        We say that $E$ is a $X$-spectral set if $\varphi\in X$ with $supp(\varphi)\subseteq E,$ then $\varphi\in I_\mathcal{A}(E)^\perp.$ It is shown in \cite[Proposition 2.4]{PaPr1} that $E$ is a set of $X$-synthesis if $I_{\mathcal{A}}^X(E)=J_{\mathcal{A}}^X(E).$
	
	\subsection{Subgroup lemma and injection theorem}
	
	   We shall begin with a notation. For a closed subset $E$ of $G^n,$ we shall denote by $VN^n_E(G)$ the w*-closure of the linear span of the set $\{ \lambda_G(x_1)\otimes\lambda_G(x_2) \otimes \cdots \otimes \lambda_G(x_n):(x_1,x_2,\ldots,x_n)\in E\}.$
	
	   The following result gives a characterization of sets on synthesis in terms of the support of an element of $VN^n(G).$
	   \begin{prop}\label{Synth_Char_Supp}
		  Let $E\subseteq G^n$ be a closed set. Then $E$ is a set of synthesis for $A^n(G)$ if and only if for any $T\in VN^n(G)$ with $supp(T)\subseteq E$ we have $T\in VN^n_E(G).$
	   \end{prop}
	   \begin{proof}
	       Suppose that $E$ is a set of synthesis for $A^n(G).$ Let $T\in VN^n(G)$ be such that $supp(T)\subseteq E.$ Then by $(ii)$ of Proposition \ref{Supp_Equi_Defn}, $\langle u,T \rangle=0$ for every $u\in A^n(G)$ with $supp(u)$ being compact and $supp(u)\cap E=\emptyset.$ Thus, by assumption, $$T\in j_{A^n(G)}(E)^\perp=J_{A^n(G)}(E)^\perp=I_{A^n(G)}(E)^\perp=VN^n_E(G)$$ and hence the proof of the forward part. We now prove the converse part. We shall prove this by assuming the contradiction. So, let us suppose that the set $E$ is not a set of synthesis for $A^n(G).$ Then $J_{A^n(G)}(E)$ is a proper closed ideal of $I_{A^n(G)}(E).$ Now, choose $u\in I_{A^n(G)}(E)\setminus J_{A^n(G)}(E).$ By Hahn-Banach theorem there exists $T\in VN^n(G)$ such that $T\in J_{A^n(G)}(E)^\perp$ and $\langle u,T \rangle\neq 0$ and hence by $(ii)$ of Proposition \ref{Supp_Equi_Defn} $supp(T)\subseteq E.$ On the other hand, $T\notin I_{A^n(G)}(E)^\perp=VN^n_E(G),$ a contradiction. 
	   \end{proof}
	   Here is the multidimensional analogue of the subgroup lemma. The proof given here is motivated by \cite[Theorem 3]{TaTa2}.
	   \begin{thm}\label{Example_Char}
	       Let $H$ be a closed subgroup of $G$ and let $T\in VN^n(G)$ Then $T\in VN^n_H(G)$ if and only if $supp(T)\subseteq H^n.$
	    \end{thm}
        \begin{proof}
	       Observe that the forward part is an easy consequence of $(v)$ of Lemma \ref{Supp_prop1} and $(i)$ of Lemma \ref{Supp_prop2}. We now prove the converse. Let us suppose that $supp(T)\subseteq H^n.$ Let $V$ be an open subset of $G^n$ such that $V=V_1\times V_2\times\cdots\times V_n,$ $V_i\subseteq G$ is open and $V=VH^n.$ Note that, for any $u\in A^n(G)\cap C_c(G^n)$ with $supp(u)\subseteq V,$ we have $supp(u\cdot T)\subseteq VH^n=V,$ by Lemma \ref{Supp_prop1}. It can be seen easily that the set of all $u\in A^n(G)\cap C_c(G^n)$ is dense in $P_VL^2(G^n)\cong P_{V_1}L^2(G)\otimes P_{V_2}L^2(G)\otimes\cdots\otimes P_{V_n}L^2(G),$ where $P_{V_i}$ denotes the projection on $L^2(G)$ defined by the characteristic function $\chi_{V_i}.$ Thus the range of $P_V$ is invariant under $T,$ i.e., $TP_V=P_VT.$ Considering $T^*,$ we have $T^*P_V=P_VT^*P_V.$ Let $\mathcal{A}$ denote the von Neumann algebra generated by the family of all such projections $P_{V_i}$ inside $L^\infty(G).$ Then $T\in \underset{1}{\overset{n}{\otimes^{\sigma h}}}\mathcal{A}^\prime$ and by \cite[Theorem 6]{TaTa1}, $T\in VN^n_H(G).$
	   \end{proof}
	   We have the following important result as a consequence of Proposition \ref{Synth_Char_Supp} and Theorem \ref{Example_Char}. 
	   \begin{cor}[Subgroup lemma]\label{SubGp_Lemma_Spectral}
	       Let $G$ be a locally compact group and $H$ a closed subgroup of $G.$ Then $H^n$ is a set of synthesis for $A^n(G).$
	   \end{cor}
	   \begin{cor}
	       Singletons are sets of spectral synthesis.
	   \end{cor}
	   Before we state our next result, here is a remark.
	   \begin{rem}
	       If $X\subseteq VN^n(G)$ is an $A^n(G)$-module, then $(r_n^*)^{-1}(X)$ is an $A^n(H)$-submodule of $VN^n(H).$ We shall denote this submodule as $X_H.$ Further, if $T\in X_H$ and $supp(T)\subseteq E\subseteq H,$ then $supp(r_n^*(T))\subseteq E\subseteq G.$
	   \end{rem}    
	   Here is the promised result on the multidimensional analogue of the injection theorem for spectral synthesis.
	   \begin{thm}[Injection theorem]\label{Inj_Thm_Sp_Syn}
	       Suppose that $H$ is a closed subgroup of a locally compact group $G$ and $X$ an $A^n(G)$-submodule of $VN^n(G).$ Then a closed subset $E$ of $H^n$ is a set of $X_H$-synthesis for $A^n(H)$ if and only if $E$ is a set of $X$-synthesis for $A^n(G).$
	   \end{thm}
	   \begin{proof}
	       We will only prove the backward part as the forward part is trivial. So, let us assume that $E$ is a set of $X_H$-synthesis concerning $A^n(H).$ We now claim that $E$ is a set of $X$-synthesis with respect to $A^n(G).$ Let $T\in X$ with $supp(T)\subseteq E.$ By Corollary \ref{SubGp_Lemma_Spectral}, $H^n$ is a set of synthesis with respect to $A^n(G)$ and hence $T\in (I_{A^n(G)}(H^n))^\perp.$ Therefore, there exists $S\in X_H$ such that $r_n^*(S)=T.$ It is clear that $supp(S)\subseteq E$ and hence, for any $u\in I(E^n),$ $$\langle T,u \rangle=\langle S,r_n(u) \rangle=0,$$ thereby proving the theorem.
	   \end{proof}
	
	\subsection{Ideals with bounded approximate identities}
	   In the final part of this section, we prove the existence of bounded approximate identities in the ideals $I_{A^n(G)}(H^n),$ where $H$ is a closed subgroup of $G.$ Our strategy here is to use a similar result for $A(G^n)$ \cite{FKLS}.
	   \begin{thm}\label{IBAI}
	       Let $G$ be an amenable locally compact group and let $E$ be a closed subset of $G$ such that $E$ is a set of spectral synthesis for $A^n(G).$ If the closed ideal $I_{A(G^n)}(E)$ has a bounded approximate identity with bound $c,$ then so has $I_{A^n(G)}(E).$
	   \end{thm}
	   \begin{proof}
	       Let $\Theta$ denote the inclusion of $A(G^n)\cong A(G)\hat{\otimes}A(G)\hat{\otimes}\cdots\hat{\otimes}A(G)$($n$-times) inside $A^n(G).$ If is clear that $\Theta(I_{A(G^n)}(E))\subseteq I_{A^n(G)}(E).$ We first claim that $\Theta(I_{A(G^n)}(E))$ is dense in $I_{A^n(G)}(E).$ Let $u\in I_{A^n(G)}(E)$ and let $\epsilon>0$ be arbitrary. Since $E$ is a set of spectral synthesis for $A^n(G)$ there exists $v\in j_{A^n(G)}(E)$ such that $\|u-v\|_{A^n(G)}<\epsilon/2.$ Choose an open set $U\subseteq G^n$ such that $supp(v)\subseteq U\subseteq E^c.$ Since $A(G^n)$ is regular, there exists $w\in A(G^n)$ such that $w|_{supp(v)}\equiv 1$ and $supp(w)\subseteq U.$ Observe that $v\Theta(w)=v.$ As $A(G^n)$ is dense in $A^n(G)$ there exists $v^\prime\in A(G^n)$ such that $\|v-\Theta(v^\prime)\|_{A^n(G)}<\frac{\epsilon}{2\|\Theta(w)\|_{A^n(G)}}.$ Note that $v^\prime w\in I_{A(G^n)}(E).$ Now, 
		  \begin{eqnarray*}
		      \|u-\Theta(v^\prime w)\|_{A^n(G)} &\leq& \|u-v\|_{A^n(G)}+\|v-\Theta(v^\prime w)\|_{A^n(G)} \\ &\leq& \frac{\epsilon}{2} + \|v\Theta(w)-\Theta(v^\prime)\Theta(w))\|_{A^n(G)} \\ &\leq& \frac{\epsilon}{2}+\|\Theta(w)\|_{A^n(G)}\|v-\Theta(v^\prime)\|_{A^n(G)} <\epsilon, 
		  \end{eqnarray*}
		  thereby proving our claim.
		
	       Suppose that $\{u_\alpha\}_{\alpha\in\wedge}$ is an approximate identity for $I_{A(G^n)}(E)$ such that $\|u_\alpha\|$ $\leq c$ for all $\alpha\in\wedge.$ Let $u\in I_{A^n(G)}(E)$ and let $\epsilon>0$ be arbitrary. Using the density of $\Theta(I_{A(G^n)}(E))$ in $I_{A^n(G)}(E),$ we can find $v\in I_{A(G^n)}(E)$ such that $\|u-\Theta(v)\|_{A^n(G)}<\frac{\epsilon}{2}.$ Since $\{u_\alpha\}$ is an approximate identity, there exists $\alpha_0\in\wedge$ such that $\|v-vu_\alpha\|_{A(G^n)}<\frac{\epsilon}{2}$ for all $\alpha\geq\alpha_0.$ Now, using the fact that $\Theta$ is a contraction (by $(iv)$ of Lemma \ref{FunctProp}), 
	       \begin{eqnarray*}
	           \|u-\Theta(u_\alpha)u\|_{A^n(G)} &\leq& \|u-\Theta(v)\|_{A^n(G)} + \|\Theta(v)-\Theta(vu_\alpha)\|_{A^n(G)} \\ &\leq& \frac{\epsilon}{2} + \|\Theta\| \|v-vu_\alpha\| <\epsilon,
		   \end{eqnarray*}
	       whenever $\alpha\geq\alpha_0,$ thereby showing that $\{\Theta(u_\alpha)\}_{\alpha\in \wedge}$ is the required approximate identity.
       \end{proof}
	   \begin{cor}\label{BAIId}
	       Let $G$ be an amenable locally compact group and let $H$ be a closed subgroup of $G.$ Then the ideal $I_{A^n(G)}(H^n)$ has a bounded approximate identity.
        \end{cor}
        \begin{proof}
	       By combining \cite[Theorem 1.5]{FKLS}, Corollary \ref{SubGp_Lemma_Spectral} and Theorem \ref{IBAI}, the proof of this follows.
        \end{proof}
	
\section{Ditkin sets}
	
	In this section, we prove the subgroup lemma, injection theorem, and inverse projection theorem for Ditkin sets.
	
	\subsection{Subgroup lemma}
        We now proceed to prove the subgroup lemma. To do this we begin with a simple lemma. The proof of this is a consequence of the fact that this lemma is true for $A(G^n)$ and the inclusion of $A(G^n)$ into $A^n(G)$ is a contraction.
        \begin{lem}
	       Given an open set $U$ containing $e\in G^n$ and an $\epsilon>0$ there are an open set $V$ of $e$ in $G^n$ and $u\in A^n(G)$ such that 
	       \begin{enumerate}[a)]
                \item $V\subseteq U,$
                \item $u(x)=1\ \forall\ x\in V,$ 
                \item $supp(u)\subseteq U$ and 
                \item $\|u\|_{A^n(G)}<1+\epsilon.$
            \end{enumerate} 
        \end{lem}
        \begin{nota}
            If $H$ is a closed subgroup of a locally compact group $G,$ then we shall denote by $q_n$ the canonical quotient map from $G^n$ onto $(G/H)^n$
        \end{nota}
        The following is an analogue of \cite[Proposition 10]{Der1}
        \begin{prop}\label{Derighetti_Prop}
            Let $H$ be a closed normal subgroup of $G.$ Given $u\in A^n(G)$ and $\epsilon,\eta>0$ there exists $v\in A^n(G/H)$ and an open set $W$ of $H^n$ with $v\circ q_n\equiv 1$ on $W$ such that $\|v\|_{A^n(G/H)}\leq 1+\eta$ and $\|u(v\circ q_n)\|_{A^n(G)}<\epsilon + \|u|_{H^n}\|_{A^n(H)}.$
        \end{prop}
        \begin{proof}
            Suppose that $u\in A^n(G)$ such that $u|_{H^n}\equiv 0$ and let $\epsilon,\eta>0.$ Since $H^n$ is a set of synthesis (Corollary \ref{SubGp_Lemma_Spectral}), there exists $u^\prime\in A^n(G)\cap C_c(G^n)$ such that $supp(u^\prime)\cap H^n=\emptyset$ and $\|u-u^\prime\|_{A^n(G)}<\frac{\epsilon}{1+\eta}.$ Choose disjoint open sets $\widetilde{U}$ and $\widetilde{V}$ of $(G/H)^n$ such that $\dot{e}:=q_n(e)\in\widetilde{U},$ $q_n(supp(u^\prime))\subseteq \widetilde{V}.$ By previous Lemma, there exists $v\in A^n(G/H)$ and an open set $\widetilde{W}$ of $\dot{e}\in(G/H)^n$ such that $\widetilde{W}\subseteq \widetilde{U},$ $v\equiv 1$ on $\widetilde{W},$ $supp(v)\subseteq\widetilde{U}$ and $\|v\|<1+\eta.$ Let $W=(q_n)^{-1}(\widetilde{W}).$ This $v$ and $W$ will satisfy the requirements.
            
            Now, let $u\in A^n(G)$ be arbitrary. Without loss of generality, let us even assume that $\|u|_{H^n}\|_{A^n(H)}\neq 0.$ There exists $w\in A^n(G)$ such that $\|w\|_{A^n(G)}=\|u|_{H^n}\|_{A^n(H)}$ and $u|_{H^n}=w|_{H^n}.$ Let $u_1=u-w.$ Then $u_1|_{H^n}\equiv 0$ and hence by the above case there exists an open set $W$ of $H^n$ and $v\in A^n(G/H)$ such that $v\circ q_n\equiv 1$ on $W,$ $\|u_1(v\circ q_n)\|_{A^n(G)}<\frac{\epsilon}{2}$ and $\|v\|_{A^n(G/H)}<1+\rho,$ where $\rho$ is a positive real number such that $\rho<\min\left\{\eta,\frac{\epsilon}{2\|u|_{H^n}\|_{A^n(H)}}\right\}.$ Then
	        \begin{align*}
	           \|(v\circ q_n)u\|_{A^n(G)} =& \|(v\circ q_n)(u_1+w)\|_{A^n(G)} \\ \leq& \|(v\circ q_n)u_1\|_{A^n(G)} + \|(v\circ q_n)w\|_{A^n(G)} \\ <& \frac{\epsilon}{2} + (1+\rho)\|u|_{H^n}\|_{A^n(H)} < \epsilon + \|u|_{H^n}\|_{A^n(H)}. \qedhere
	        \end{align*}
        \end{proof}
        Our next result is the subgroup lemma for local Ditkin sets, where the subgroups are normal.
        \begin{thm}\label{SubGp_Lemma_Ditkin}
            Let $H$ be a closed normal subgroup of $G.$ Then $H^n$ is a locally Ditkin set for $A^n(G).$
        \end{thm}
        \begin{proof}
            Given Proposition \ref{Derighetti_Prop}, the proof of this follows the same lines as in \cite[Th\'eor\`eme 11]{Der1}.
        \end{proof}
        As an immediate corollary, we obtain the following.
        \begin{cor}
            Singletons are local Ditkin sets.
        \end{cor}
	
	\subsection{Injection theorem}
	We now proceed to prove the injection theorem. We shall begin with a lemma which gives one part of the injection theorem.
	\begin{lem}\label{DSInjThmFP}
		Let $X$ be an $A^n(G)$-submodule of $VN^n(G)$ and let $H$ be a closed subgroup of $G.$ If $E\subseteq H^n \subseteq G^n$ is a closed subset then $E$ is an $X$-Ditkin set for $A^n(G)$ implies that $E$ is an $X_H$-Ditkin set for $A^n(H).$
	\end{lem}
	\begin{proof}
		Suppose that $E$ is an $X$-Ditkin set for $A^n(G).$  In order to prove that $E$ is an $X_H$-Ditkin set for $A^n(H),$ it is enough to show that for any given $u\in A^n(H)$ with $u(x)=0$ for all $x\in E$ and $S\in X_H,$ there exists $v\in A^n(H)$ such that $v$ vanishes in an open set $U$ with $E\subseteq U\subseteq H^n$ and $\langle u,S \rangle = \langle uv,S \rangle.$ So, let $u\in A^n(H)$ such that $u|_E\equiv 0$ and let $S\in X_H.$ Since $S\in X_H,$ $r_n^*(S)\in X.$ Choose $u_1\in A^n(G)$ such that such that $u_1$ extends $u.$ Now, by assumption, $E$ is an $X$-Ditkin set for $A^n(G).$ Therefore, there exists $v_1\in A^n(G)$ such that $v_1(x)=0$ for all $x$ in some open set $U_1$ such that $E\subseteq U_1\subseteq G^n$ and $\langle u_1,r_n^*(S) \rangle = \langle u_1v_1,r_n^*(S) \rangle.$ Now, 
		\begin{eqnarray*}
			\langle u,S \rangle &=& \langle r_n(u_1),S \rangle = \langle u_1,r_n^*(S) \rangle \\ &=& \langle v_1u_1,r_n^*(S) \rangle = \langle r_n(v_1u_1),S \rangle = \langle r_n(v_1)u,S \rangle,
		\end{eqnarray*}
		thereby proving the Lemma.
	\end{proof}
	Here is the injection theorem under the assumption that the subgroup is normal.
	\begin{thm}[Injection theorem for Ditkin sets]\label{Inj_Thm_Ditkin_1}
		Let $H$ be a closed normal subgroup of $G$ and let $E$ be a closed subset of $H^n.$ Then $E$ is a local Ditkin set for $A^n(G)$ if and only if $E$ is a local Ditkin set for $A^n(H).$
	\end{thm}
	\begin{proof}
		The proof of the forward part is a direct consequence of Lemma \ref{DSInjThmFP}. So we need to prove only the converse part. But, by using Proposition \ref{Derighetti_Prop} and Theorem \ref{Inj_Thm_Sp_Syn}, one can repeat the proof of Th\'eor\`eme 12 of \cite{Der1} and hence the details are omitted.
	\end{proof}
	We now prove the injection theorem without the assumption of normality on the subgroups. To begin with, here is a simple lemma.
	\begin{lem}\label{SpeLemma}
		Let $H$ be a closed subgroup of an amenable locally compact group $G.$ Given $u\in A^n(G)$ and $\epsilon>0,$ there exists $v\in A^n(G)$ such that $v$ vanishes in a neighbourhood of $H^n$ and $\|u-uv\|\leq 2\|u|_{H^n}\|+\epsilon.$
	\end{lem}
	\begin{proof}
		Let $u\in A^n(G).$ By $(iv)$ of Lemma \ref{FunctProp} there exists $w\in I_{A^n(G)}(H^n)$ such that $$\|u-w\|_{A^n(G)}\leq \|u|_{H^n}\|_{A^n(H)}+\epsilon/4.$$ By Corollary \ref{SubGp_Lemma_Spectral}, $H^n$ is a set of spectral synthesis and also as $G$ is amenable, by Corollary \ref{BAIId}, the ideal $I_{A^n(G)}(H^n)$ has an approximate identity bounded by 1. In particular, there exists $v\in j_{A^n(G)}(H^n)$ such that $\|v\|\leq 1$ and $$\|w-vw\|\leq \epsilon/4.$$ This $v$ satisfies the requirements of the Lemma.
	\end{proof}
	Here is the injection theorem without any assumption on the subgroup. But the assumption is on $G,$ we assume that the group is amenable.
	\begin{thm}[Injection theorem for Ditkin sets]\label{Inj_Thm_Ditkin_2}
		Let $H$ be a closed subgroup of an amenable locally compact group $G$ and let $E$ be a closed subset of $H^n.$ Then $E$ is a Ditkin set for $A^n(G)$ if and only if $E$ is a Ditkin set for $A^n(H).$
	\end{thm}
	\begin{proof}
		The forward part is a special case of Lemma \ref{DSInjThmFP}. The backward is almost the same as in Th\'eor\`eme 12 of \cite{Der1}, once we use Lemma \ref{SpeLemma} and hence we omit it.
	\end{proof}
	
	\subsection{Inverse projection theorem}
	
	In this final part, we prove the inverse projection theorem. We shall begin with a lemma which gives one part without any assumptions on $G.$
	
	\begin{lem}\label{IPTConv}
		Let $H$ be a closed normal subgroup of a locally compact group $G$ and let $\widetilde{E}$ be a closed subset of $(G/H)^n.$ Then 
		\begin{enumerate}[(i)]
			\item $\widetilde{E}$ is a set of synthesis for $A^n(G/H)$ if $q_n^{-1}(\widetilde{E})$ is a set of spectral synthesis for $A^n(G).$
			\item $\widetilde{E}$ is a strong Ditkin set for $A^n(G/H)$ if $q_n^{-1}(\widetilde{E})$ is a strong Ditkin set for $A^n(G).$
		\end{enumerate}
	\end{lem}
	\begin{proof}
		$(i)$ Suppose that $q_n^{-1}(\widetilde{E})$ is a set of synthesis for $A^n(G).$ Let $\widetilde{u}\in I_{A^n(G/H)}(\widetilde{E})$ and $u=\widetilde{u}\circ q_n$. Obsevre that $u \in I_{A^n(G)}(q_n^{-1}(\widetilde{E})).$ Since $q_n^{-1}(\widetilde{E})$ is a set of spectral synthesis with respect to $A^n(G)$, $u \in J_{A^n(G)}(q_n^{-1}(\widetilde{E}))$ and hence $\widetilde{u}\in J_{A^n(G/H)}(\widetilde{E})$, thereby proving that $\widetilde{E}$ is a set of spectral synthesis for $A^n(G/H)$.
		
		$(ii)$ Suppose that $q_n^{-1}(\widetilde{E})$ is a strong Ditkin set for $A^n(G)$. Let $\widetilde{u} \in I_{A^n(G/H)}(\widetilde{E}).$ Then $u =\widetilde{u} \circ q_n \in I_{A^n(G)}(q_n^{-1}(\widetilde{E})).$ Since $q_n^{-1}(\widetilde{E})$ is a strong Ditkin set, there exists a bounded sequence $\{u_n\} \subseteq j_{A^n(G)}(q_n^{-1}(\widetilde{E}))$ such that $$\| u_nu-u\| \rightarrow 0.$$ Observe that $P_{H^n}(u_n) \in j_{A^n(G/H)}(\widetilde{E}))$ and therefore $$ \| P_{H^n}(u_n)\widetilde{u}_n - \widetilde{u}_n \| \rightarrow 0, $$ thereby showing that $\widetilde{E}$ is a strong Ditkin set for $A^n(G/H)$.
	\end{proof}
	Here is the complete inverse projection theorem under the assumption that the group is compact.
	\begin{thm}[Inverse projection theorem]\label{IPT}
		Let $H$ be a closed normal subgroup of a compact group $G$ and let $\widetilde{E}$ be a closed subset of $(G/H)^n.$ Then 
		\begin{enumerate}[(i)]
			\item $\widetilde{E}$ is a set of synthesis for $A^n(G/H)$ if and only if $q_n^{-1}(\widetilde{E})$ is a set of spectral synthesis for $A^n(G).$
			\item $\widetilde{E}$ is a strong Ditkin set for $A^n(G/H)$ if and only if $q_n^{-1}(\widetilde{E})$ is a strong Ditkin set for $A^n(G).$
		\end{enumerate}
	\end{thm}
	\begin{proof}
		$(i)$ We shall only prove the forward part, as the backward part follows from Lemma \ref{IPTConv}. So, let $v \in I_{A^n(G)}(q_n^{-1}(\widetilde{E})).$ For each irreducible unitary representation $\pi = \underset{1}{\overset{n}{\otimes}}\pi^i$ of $G^n,$ consider the matrix-valued functions $v^{\pi}$ and $\widetilde{v}^{\pi}$ given by, $$v^{\pi}(x_1,x_2,...,x_n) := \underset{H^n}{\int} (\underset{1}{\overset{n}{\otimes}}\pi^i(h_i)) v(x_1 h_1,...,x_n h_n) dh_1...dh_n,$$ $$\widetilde{v}^{\pi}(x_1,x_2,...,x_n) := \underset{1}{\overset{n}{\otimes}} \pi^i(x_i)v^{\pi}(x_1,...x_n).$$ Observe that $\widetilde{v}^{\pi}$ satisfies $$ \widetilde{v}^{\pi}(x_1h_1,...,x_nh_n) = \widetilde{v}^{\pi}(x_1,...,x_n), \forall (x_1,...,x_n) \in G^n, (h_1,...,h_n) \in H^n,$$ i.e., $\widetilde{v}^{\pi}$ is invariant under translation by elements of $H^n.$ Note that the matrix coefficients of $\pi$ are of the form $$ \underset{1}{\overset{n}{\otimes}}\pi_{k_i,l_i}(x_1,...x_n) =\underset{1}{\overset{n}{\prod}}\pi^i_{k_i,l_i}(x_i), 1\leq k_i,l_i \leq d_{\pi^i}.$$         Therefore, $$ \underset{1}{\overset{n}{\otimes}} v^{\pi^i}_{k_i,l_i}(x_1,...x_n) = \underset{H^n}{\int}( \underset{1}{\overset{n}{\otimes}}\pi_{k_i,l_i}(x_1,...x_n)) v(x_1h_1,..., x_nh_n)dh_1...dh_n. $$ Hence
		\begin{eqnarray*}
			\underset{1}{\overset{n}{\otimes}} \widetilde{v}^{\pi^i}_{k_i,l_i}(x_1,...x_n) &=& \left(\underset{1}{\overset{n}{\prod}}\pi^i_{k_i,l_i}(x_i)\right) \left(\underset{1}{\overset{n}{\otimes}} v^{\pi^i}_{k_i,l_i}\right)(x_1,...x_n)\\ &=& \left(\underset{i=1}{\overset{n}{\prod}} \underset{m_i=1} {\overset{d_{\pi^i}}\sum}\pi^i_{k_i,m_i}(x_i)  v_{m_i,l_i}^{1 \otimes...\otimes \pi_k^i \otimes... \otimes 1}\right) \left(x_1,...x_n\right)
		\end{eqnarray*}
		Also, as $\underset{1}{\overset{n}{\otimes}} \widetilde{v}^{\pi^i}_{k_i,l_i}$ can be considered as a function in $A^n(G/H)$, $$\underset{1}{\overset{n}{\otimes}}  \widetilde{v}^{\pi^i}_{k_i,l_i} \in I_{A^n(G/H)}(\widetilde{E}).$$ Since $\widetilde{E}$ is a set of spectral synthesis for $A^n(G/H)$, $\underset{1}{\overset{n}{\otimes}} \widetilde{v}^{\pi^i}_{k_i,l_i} \in J_{A^n(G/H)}(\widetilde{E})\subseteq J_{A^n(G)}(q_n^{-1}(\widetilde{E})).$ Now, observe that $$\underset{1}{\overset{n}{\otimes}} v^{\pi^i}_{k_i,l_i}=\left(\underset{1}{\overset{n}{\prod}}\check{\pi}^i_{k_i,l_i}(x_i)\right) \left(\underset{1}{\overset{n}{\otimes}} \widetilde{v}^{\pi^i}_{k_i,l_i}\right),$$ which implies that $\underset{1}{\overset{n}{\otimes}} v^{\pi^i}_{k_i,l_i} \in J_{A^n(G)}(q_n^{-1}(\widetilde{E})).$ Choose a bounded approximate identity $\{e_\alpha\}$ in $L^1(G^n)$ such that $$e_{\alpha}.v \in span\left\{ \underset{1}{\overset{n}{\otimes}} v^{\pi^i}_{k_i,l_i}: \pi^i \in \hat{G}, 1\leq i \leq n, 1\leq k_i,l_i \leq d_{\pi^i}\right\}.$$ Note that $e_{\alpha}.v \in J_{A^n(G)}(q_n^{-1}(\widetilde{E}).$ Also, $e_{\alpha}.v \rightarrow v$, which implies that $$v \in J_{A^n(G)}(q_n^{-1}(\widetilde{E})).$$ Hence the proof.
		
		$(ii)$ As mentioned above, we shall prove only the forward part. So, let us suppose that $\widetilde{E}$ is a strong Ditkin set for $A^n(G/H)$. Let $v \in I_{A^n(G)}(q_n^{-1}(\widetilde{E})).$ With notations as in the proof of $(i)$ $\underset{1}{\overset{n}{\otimes}} v^{\pi^i}_{k_i,l_i} \in I_{A^n(G)}(q_n^{-1}(\widetilde{E})).$ Since $\widetilde{E}$ is a strong Ditkin set for $A^n(G/H)$ there exists a bounded sequence $\{\widetilde{u}_n\} \subseteq j_{A^n(G/H)}(\widetilde{E})$ such that $$ \left\| \widetilde{u}_n\left(\underset{1}{\overset{n}{\otimes}} v^{\pi^i}_{k_i,l_i} \right) -\left(\underset{1}{\overset{n}{\otimes}} v^{\pi^i}_{k_i,l_i}\right)\right\|_{A^n(G/H)} \rightarrow 0.$$ Let $u_n = \widetilde{u}_n \circ q_n.$ Then $u_n \in A^n(G)$. Choose a bounded approximate identity $\{e_\alpha\}$ for $L^1(G^n)$ such that $ e_\alpha.v = \underset{n}{\lim}(u_n(e_\alpha.v))$ in $A^n(G).$
		
		Also, $e_\alpha.v \rightarrow v $ for all $v \in A^n(G)$. Let $\|u_n\| \leq c$ for all $n \in \mathbb{N}$. Thus, 
		\begin{eqnarray*}
			& & \|u_nv-v\|_{A^n(G)} \\ &\leq& \|u_nv-u_ne_\alpha.v\|_{A^n(G)}+ \|u_ne_\alpha.v-e_\alpha.v\|_{A^n(G)}+\|e_\alpha.v-v\|_{A^n(G)}\\ &\leq& (c+1)\|e_\alpha.v-v\|_{A^n(G)} + \|u_nv-u_ne_\alpha.v\|_{A^n(G)} 
		\end{eqnarray*}
		Given $\epsilon >0$, fix an $\alpha$ such that $$\| v-e_\alpha.v\|<\frac{\epsilon}{2(1+c)}.$$ For this $\alpha$, there exists $n_0$ such that $$ \| u_ne_\alpha.v-e_\alpha.v\| < \frac{\epsilon}{2} \hspace{1cm} \forall \hspace{0.15cm} n \geq n_0.$$ Thus $u_nv \rightarrow v$. Hence $q_n^{-1}(\widetilde{E})$ is a strong Ditkin set for $A^n(G)$.
	\end{proof}
	
	\section{Synthesis in $A^n(G)$ and $A^{n+1}(G)$}
	
	The main aim of this section is to prove the Malliavin's theorem. Our strategy is to make use of Malliavin's theorem that is already available for $A(G),$ when $G$ is abelian.
	
        \subsection{Spectral sets and Ditkin sets w.r.t. $A^n(G)$ and $A^{n+1}(G)$}
	We first prove a result on a parallel synthesis between $A^n(G)$ and $A^{n+1}(G).$ Here we assume that $G$ is compact.
	
	We shall begin with a notation.
	\begin{nota}
		We let $$ A^{n+1}_{inv}(G) = \left\{u \in A^{n+1}(G): \begin{array}{c} u(x_1,x_2,\ldots,x_{n+1})=u(x_1x,x_2x,\ldots,x_{n+1}x) \\ \forall\ x,x_1,x_2,\ldots,x_{n+1}\in G\end{array} \right\}.$$
	\end{nota}
	\begin{lem}
		The mapping $P: A^{n+1}(G) \rightarrow A^{n+1}(G)$ given by $$ Pu(x_1,x_2,...,x_{n+1}) = \underset{G}{\int} u(x_1x,x_2x,...,x_{n+1}x) dx$$ is a contractive projection onto $A^{n+1}_{inv}(G)$ and an $A^{n+1}_{inv}(G)$- module map.
	\end{lem}
	\begin{proof}
		Since $G$ is compact, for every $x \in G,$ the map $u\mapsto x.u$ is an isometry on $A^{n+1}(G).$ Thus, $\|P(u)\|_{A^{n+1}(G)} < \infty,$ i.e., the $P$ map is well defined. Now, the constant function $1$ belongs to $L^1(G)$ and it is clear that $P(u)=1\cdot u$ for all $u\in A^{n+1}(G).$ Hence the map $P$ is a contraction. Using the fact that the Haar measure is left invariant, range$(P)$ sits inside $A_{inv}^{n+1}(G).$ Now, from the definition of $A^{n+1}_{inv}(G),$ it is easy to verify that $P|_{A^{n+1}_{inv}(G)}$  is the identity mapping. Thus, $P$ is a contractive projection of $A^{n+1}(G)$ onto $A^{n+1}_{inv}(G).$ Finally, it is easy to observe from the definition of $P$ that it is an $A^{n+1}_{inv}(G)$-module map.
	\end{proof}
\begin{prop}\label{EquiCondUnitMDA}
		Let $u$ be any function on $G^n.$ Then TFAE:
		\begin{enumerate}[(i)]
			\item The element $u\in b_1(A^n(G));$
			\item There exists an operator $M_u\in b_1\left( \mathcal{B}^\sigma\left( \otimes_n^{\sigma h}\ VN(G), VN(G) \right) \right)$ such that $$M_u(\lambda(x_1)\otimes\cdots\otimes\lambda(x_n))=u(x_1,\ldots,x_n)\lambda(x_1\ldots x_n);$$
			\item There exists an operator $\widetilde{M_u}\in b_1\left( \mathcal{B}^\sigma\left( \otimes_n^h\ C^\ast(G), C^\ast(G) \right) \right)$ such that $$\widetilde{M_u}(\lambda(f_1)\otimes\lambda(f_2)\otimes\cdots\otimes\lambda(f_n))=\lambda(g),$$ where $$g(x)=\underset{G^{n-1}}{\int}f_1(x_1)f_2(x_1^{-1}x_2)\cdots f_n(x_{n-1}^{-1}x)u(x_1,x_1^{-1}x_2,\ldots,x_{n-1}^{-1}x)\ dx_1\ldots dx_{n-1}.$$
		\end{enumerate}
	\end{prop}
	\begin{proof}
		$(i)\ \Rightarrow\ (ii).$ Define $\theta:A(G)\rightarrow A^n(G)$ as $$\theta(v)(x_1,x_2,\ldots,x_n):=v(x_1x_2\ldots x_n).$$ By \cite[Proposition 5.2]{ToTu},  the map $\theta$ is well defined and is a complete contraction. Define $m_u:A(G)\rightarrow A^n(G)$ as $m_u(v):=u\theta(v).$ Now, for any $v\in A^n(G)$, we have 
		\begin{eqnarray*}
			\|m_u(v)\|_{A^n(G)} &=& \|u\theta(v)\|_{A^n(G)} \leq \|u\|_{A^n(G)}\|\theta(v)\|_{A^n(G)} \\ &<& \|\theta(v)\|_{A^n(G)} \leq \|v\|_{A^n(G)},
		\end{eqnarray*}
		i.e., $\|m_u\|\leq 1$ or in other symbols $m_u\in b_1(\mathcal{B}(A(G),A^n(G))).$ Let $M_u$ denote the Banach space adjoint of $m_u.$ Then $M_u$ satisfies $(ii).$
	
		$(ii)\ \Rightarrow\ (iii).$ Let $\widetilde{M_u}=M_u|_{\otimes_n^h C^\ast(G)}.$ Then, by \cite[Pg. 19]{ToTu}, $(iii)$ follows.
			
		$(iii)\ \Rightarrow\ (i).$ By \cite{ToTu}, the spaces $A^n(G)$ and $\left( \otimes_n^h C^\ast(G) \right)^\ast$ are completely isometric to each other. Therefore, let $m_u$ be the Banach space adjoint of $\widetilde{M_u}.$ Then $m_u$ satisfies $(i).$
	\end{proof}
	\begin{prop}
	    The map $N:A^{n}(G) \rightarrow A^{n+1}_{inv}(G)$ given by $$Nu(x_1,x_2,...,x_{n+1})     = u(x_1x_2^{-1},x_2x_3^{-1},...,x_nx_{n+1}^{-1})$$ is an isometry.
	\end{prop}
	\begin{proof}
		Let $u\in A^n(G).$ By \cite[Theorem 4.1]{ToTu} and the fact that $G$ is compact, $u$ can be written as $$u(x_1,x_2,\ldots,x_n)=\left\langle \lambda(x_1)\cdots \lambda( x_n)f,g \right\rangle,$$ for some $f,g\in L^2(G).$ In fact, one can choose $f$ and $g$ such that $$\|u\|_{A^n(G)}=\|f\|_2\|g\|_2.$$ Therefore, $Nu(x_1,x_2,\ldots,x_{n+1})=\left\langle \lambda(x_{n+1})^\ast f,\lambda(x_1)^\ast g \right\rangle,$ a consequence of the definition of $N$.
		
		Now, for any $\pi\in\widehat{G}$ and $1\leq i\leq d_\pi^2,$ let $$\phi_{\pi,i}(s)=\langle e_i^\pi , \lambda(s)^\ast g \rangle\mbox{ and }\psi_{\pi,i}(s)=\langle \lambda(s)^\ast f, e_i^\pi \rangle,\ s\in G.$$ As a consequence of Parseval's formula we have $$Nu(x_1,x_2,\ldots,x_n)=\underset{\pi\in\widehat{G}}{\sum}\underset{i=1}{\overset{d_\pi^2}{\sum}}\phi_{\pi,i}(x_1)\psi_{\pi,i}(x_{n+1}).$$ Thus, $Nu = \underset{\pi\in\widehat{G}}{\sum}\underset{i=1}{\overset{d_\pi^2}{\sum}} \phi_{\pi,i} \otimes 1 \otimes \cdots \otimes  \psi_{\pi,i}$.
		
		As in the proof of \cite[Theorem 2.2]{SprT}, one can show that both  $$\underset{\pi\in\widehat{G}}{\sum}\underset{i=1}{\overset{d_\pi^2}{\sum}}|\phi_{\pi,i}|^2\mbox{ and } \underset{\pi\in\widehat{G}}{\sum}\underset{i=1}{\overset{d_\pi^2}{\sum}}|\psi_{\pi,i}|^2$$ converge uniformly in $C(G).$ But, it can be seen that $Nu\in  \otimes_{n+1}^{h} A(G).$ Further, $\|Nu\|_{\otimes_{n+1}^{h} C(G)}\leq \|Nu\|_{\otimes_{n+1}^{h} A(G)}.$ Since $\otimes^{h}_{n+1} A(G) \subseteq  \otimes^{eh}_{n+1} A(G)$ \cite[Theorem 5.3]{ER2} isometrically, the mapping $N$ is a contraction.
	
		Conversely, let $u \in A^{n+1}_{inv}(G).$ Define $$T_u:\otimes_n^{\sigma h}\ \mathcal{B}(L^2(G))\rightarrow\mathcal{B}(L^2(G))$$ as $$T_u(S_1\otimes S_2\otimes\cdots \otimes S_{n})=\underset{i=1}{\overset{\infty}{\sum}}\ M_{\phi^1_i}S_1M_{\phi^2_i}S_2M_{\phi^3_i}\cdots S_nM_{\phi^{n+1}_i},$$ where $u=\underset{i=1}{\overset{\infty}{\sum}} \phi^1_i\otimes\phi^2_i\otimes\cdots\otimes\phi^{n+1}_i.$ Note that $ T_u(\lambda(x_1)\otimes\cdots \otimes\lambda(x_n))(f)(x_{n+1})$ $=$ $u(e,x_1^{-1},\ldots,x_n^{-1}x_{n-1}^{-1}\ldots x_2^{-1} x_1^{-1}) \lambda(x_1x_2\ldots x_n)f(x_{n+1}).$ Thus $T_u(\otimes_n^{\sigma h}VN(G))\subseteq VN(G).$ For $v \in A^n(G),$ let $$M_v(\lambda(x_1) \otimes \cdots \otimes \lambda(x_n)) = v(x_1,...,x_n) \lambda(x_1x_2...x_n).$$ Then $M_{N^{-1}(u)} = T_u \restriction_{\otimes_n^{\sigma h}VN(G)}.$ Thus $M_{N^{-1}(u)}$ satisfies $(ii)$ of  Proposition \ref{EquiCondUnitMDA} and therefore $\|M_{N^{-1}(u)}\|\leq\|T_u\|\leq\|u\|_{eh}.$ Therefore, $N^{-1}u \in A^n(G).$ In particular, the map $N^{-1}$ is bounded, in fact a contraction.
	\end{proof}
        \begin{lem}\label{RelIdeals1}
			Let $E$ be a closed subset of $G^n$ and let $u\in A^n(G).$
			\begin{enumerate}[(i)]
				\item $u\in I_{A^n(G)}(E)$ if and only if $Nu\in I_{A^{n+1}(G)}(E^\ast).$
				\item $u\in J_{A^n(G)}(E)$ if and only if $Nu\in J_{A^{n+1}(G)}(E^\ast).$
			\end{enumerate}
		\end{lem}
        \begin{proof}
            The proof follows a similar approach as in \cite[Theorem 3.1]{SprT}.
        \end{proof}
	
	Here is the result of parallel synthesis.
	\begin{thm}\label{Syn_An_Anplus1}
		Assume that $G$ is compact. Let $E \subseteq G^n$ be closed and let $$E^* = \{(x_1,x_2,...,x_{n+1}) :(x_1x_2^{-1}, x_2x_3^{-1},..., x_nx_{n+1}^{-1}) \in E\}.$$
		(i) The set $E$ is a set of spectral synthesis for $A^n(G)$ if and only if $E^*$ is a set of spectral synthesis for $A^{n+1}(G)$.\\
		(ii) The set $E$ is a strong Ditkin set for $A^n(G)$ if and only if $E^*$ is a strong Ditkin set for $A^{n+1}(G)$.
	\end{thm}
	\begin{proof}
    (i)	Note that the backward part is an easy consequence of Lemma \ref{RelIdeals1}. We shall now prove the forward part. Here we closely follow Spronk and Turowska \cite[Theorem 3.1]{SprT}. Suppose that $E$ is a set of synthesis for $A^n(G).$ Let $\omega\in I_{A^{n+1}(G)}(E^\ast).$ Consider a continuous irreducible unitary representation $\pi$ of $G$ with $\mathcal{H}_\pi$ as the representation space. Define two functions $\omega^\pi$ and $\widetilde{\omega}^\pi$ on $G^{n+1}$ with range in $\mathcal{H}_\pi,$ as $$\omega^\pi(x_1,x_2,\ldots,x_{n+1})=\int_G \omega(x_1x,x_2x,\ldots,x_{n+1}x)\pi(x)\ dx$$ and $$\widetilde{\omega}^\pi(x_1,x_2,\ldots,x_{n+1})=\pi(x_1)\omega^\pi(x_1,x_2,\ldots,x_{n+1}).$$ Since $\mathcal{H}_\pi$ is finite dimensional, let us fix an orthonormal basis and write $u_{ij}^\pi$ for the matrix coefficient functions of $\pi.$ Also, let $\omega^\pi_{ij}=u_{ij}^\pi\omega^\pi$ and $\widetilde{\omega}^\pi_{ij}=u_{ij}^\pi\widetilde{\omega}^\pi.$ Note that $\omega^\pi_{ij}\in I_{A^{n+1}(G)}(E^\ast)$ and hence $\widetilde{\omega}^\pi_{ij}=\underset{k=1}{\overset{d_\pi}{\sum}}(u_{ik}^\pi\otimes 1\otimes\cdots\otimes 1)\omega_{kj}^\pi\in I_{A^{n+1}(G)}(E^\ast).$ Observe that $$\widetilde{\omega}^\pi(x_1x,x_2x,\ldots,x_{n+1}x)=\widetilde{\omega}^\pi(x_1,x_2,\ldots,x_{n+1})$$ and hence $\widetilde{\omega}_{ij}^\pi\in A^{n+1}_{inv}(G).$ Thus, by Lemma \ref{RelIdeals1}, $N^{-1}(\widetilde{\omega}_{ij}^\pi)\in I_{A^n(G)}(E)$ and by the assumption on $E,$ we have $N^{-1}(\widetilde{\omega}_{ij}^\pi)\in J_{A^n(G)}(E).$ Applying Lemma \ref{RelIdeals1} once again, we have that $\widetilde{\omega}_{ij}^\pi\in J_{A^{n+1}(G)}(E^\ast.)$ Since $$\omega^\pi(x_1,x_2,\ldots,x_{n+1})=\pi(x^{-1})\widetilde{\omega}^\pi(x_1,x_2,\ldots,x_{n+1}),$$ we have $$\omega_{ij}^\pi=\underset{k=1}{\overset{d_\pi}{\sum}}(\check{u}_{ik}^\pi\otimes 1\otimes\cdots\otimes 1)\widetilde{\omega}_{kj}^\pi\in J_{A^{n+1}(G)}(E^\ast),$$ where $\check{u}(x)=u(x^{-1}).$ Using the fact that $A^{n+1}(G)$ is an essential $L^1(G)$-module, the proof is now completed by using an approximate identity argument (as in \cite[Theorem 3.1]{SprT}) to conclude that $\omega\in J_{A^{n+1}(G)}(E^\ast.)$ Thus $E^\ast$ is a set of spectral synthesis for $A^{n+1}(G).$\\
    (ii) Suppose that $E$ is a strong Ditkin set for $A^n(G).$ Let $\{u_m\}$ be a bounded sequence in $j_{A^n(G)}(E)$ such that $u_mu\rightarrow u$ for all $u\in I_{A^n(G)}(E).$ Let $\omega\in I_{A^{n+1}(G)}(E^\ast).$ For each $\pi\in\widehat{G},$ let $\omega^\pi,$ $\omega^\pi_{ij}\in I_{A^{n+1}(G)}(E^\ast)\cap A_{inv}^{n+1}(G),$ as above. Then $N^{-1}\omega^\pi_{ij}\in I_{A^n(G)}(E)$ and hence $u_m N^{-1}\omega^\pi_{ij}\rightarrow N^{-1}\omega^\pi_{ij}.$ Applying $N,$ $Nu_m\omega^\pi_{ij}\rightarrow \omega^\pi_{ij}.$ As mentioned above, choose an approximate identity $\{f_\alpha\}$ for $L^1(G).$ Thus, for each $\alpha,$ $$Nu_m(f_\alpha\cdot\omega)\rightarrow f_\alpha\cdot\omega.$$ Now, 
			\begin{eqnarray*}
				\|Nu_m\omega - \omega\|_{A^{n+1}(G)} &\leq& \|Nu_m\omega - Nu_m(f_\alpha\cdot\omega)\|_{A^{n+1}(G)} \\ & & + \|Nu_m(f_\alpha\cdot\omega) - f_\alpha\cdot\omega\|_{A^{n+1}(G)} + \|f_\alpha\cdot\omega - \omega\|_{A^{n+1}(G)} \\ &\leq& \|Nu_m\|_{A^{n+1}(G)}\|f_\alpha\cdot\omega - \omega\|_{A^{n+1}(G)} \\ & & + \|Nu_m(f_\alpha\cdot\omega) - f_\alpha\cdot\omega\|_{A^{n+1}(G)} + \|f_\alpha\cdot\omega - \omega\|_{A^{n+1}(G)} \\ &\leq& C\|f_\alpha\cdot\omega - \omega\|_{A^{n+1}(G)} + \|Nu_m(f_\alpha\cdot\omega) - f_\alpha\cdot\omega\|_{A^{n+1}(G)},
			\end{eqnarray*}
			where $C$ is a bound for the sequence $\{Nu_m\}.$ Now, fix an $\alpha$ so that the first term is small and for this $\alpha,$ choose $m$ large enough so that the second term is also small. As the sequence $\{Nu_m\}$ is bounded, $E^\ast$ is a strong Ditkin set.
			
			We now prove the converse. This part, as we will see, easily follows from Lemma \ref{RelIdeals1}. Suppose that $E^\ast$ is a strong Ditkin set for $A^{n+1}(G).$ Let $\{\omega_m\}$ be a bounded sequence in $j_{A^{n+1}(G)}(E^\ast)$ such that $$\|\omega_m\omega-\omega\|_{A^{n+1}(G)}\rightarrow 0 \mbox{ as }m\rightarrow\infty\mbox{ and }\forall\ \omega\in I_{A^{n+1}(G)}(E^\ast).$$ Let $u\in I_{A^n(G)}(E).$ Then, by Lemma \ref{RelIdeals1}, $Nu\in I_{A^{n+1}(G)}(E^\ast)$ and hence $$\|\omega_m Nu-Nu\|\rightarrow\mbox{ as }m\rightarrow\infty.$$ Using the fact that $P^n$ is continuous and also a $V_{inv}^n(G)$-module map, we have that $$P(\omega_m Nu)\rightarrow Nu\mbox{ as }m\rightarrow\infty.$$ Hence $N^{-1}(P(\omega_m Nu))\rightarrow u$ in $A^n(G).$ Observe that the sequence $\{\omega_m\}$ is contained in $j_{A^{n+1}(G)}(E^\ast)$ and hence the sequence $\{N^{-1}P\omega_m\}$ belongs to $j_{A^n(G)}(E),$ thereby showing that $(N^{-1}P\omega_m)u\rightarrow u$ in $J_{A^n(G)}(E).$ Thus $E$ is a strong Ditkin set and hence the proof of the theorem.
		\end{proof}
	\subsection{Application to Malliavin's theorem}
	
	In this final part of the paper, we prove the multidimensional analogue of Malliavin's theorem. 
	
	We shall begin with a corollary. Although this generalizes the case when $n=1,$ it makes use of the same.
	\begin{cor}\label{MallThm1}
		If $G$ is an infinite compact group then there exits a closed subset $E$ of $G^n$ such that $E$ is not a set of synthesis for $A^n(G).$
	\end{cor}
	\begin{proof}
		This follows from Theorem \ref{Syn_An_Anplus1} and the existence of a set of nonsynthesis for $A(G)$ \cite[Proposition 2.2]{KL}.
	\end{proof}
	We now prove Malliavin's theorem for abelian groups.
	\begin{cor}\label{MallThm2}
		Let $G$ be a locally compact abelian group. Then every closed subset of $G^n$ is a set of spectral synthesis for $A^n(G)$ if and only if $G$ is discrete.
	\end{cor}
	\begin{proof}
		This follows from the structure theorem for abelian groups \cite[Theorem 2.4.1]{Ru}, Lemma \ref{IPTConv}, Theorem \ref{Inj_Thm_Sp_Syn} and Corollary \ref{MallThm1}.
	\end{proof}
	Before we proceed to the main theorem, here is a simple lemma.
	\begin{lem}\label{clopen}
		If $E$ is a clopen subset of $G^n$ such that $u\in\overline{uA^n(G)}$ for each $u\in I_{A^n(G)}(E),$ then $E$ is a set of spectral synthesis.
	\end{lem}
	\begin{proof}
		Let $E$ be a clopen subset of $G^n$ such that $u\in\overline{uA^n(G)}.$ To show that $E$ is a set of spectral synthesis, we make use of Proposition \ref{Synth_Char_Supp}. So, let $T\in VN^n(G)$ such that $supp(T)\subseteq E$ and let $u\in I_{A^n(G)}(E).$ By assumption, there exists a net $\{u_\alpha\}$ contained in $A^n(G)$ such that $uu_\alpha\rightarrow u.$ In fact, we can choose these $u_\alpha$'s such that $supp(u_\alpha)$ is compact. Then, as $E$ is a clopen set, $u_\alpha \cdot t=0$ and hence $u\cdot T=0.$ Now, by regularity of $A^n(G)$ there exists $v\in A^n(G)$ such that $v(x)=1$ for all $x\in supp(u).$ Thus,
		\begin{align*}
			\langle u,T \rangle &= \langle \lim uu_\alpha,T \rangle = \lim \langle uu_\alpha,T \rangle = \lim \langle u_\alpha,u\cdot  T \rangle \qedhere
		\end{align*}
	\end{proof}
	Here is the final result of this paper on the multidimensional analogue of Malliavin's theorem on the failure of spectral synthesis in non-discrete groups.
	\begin{thm}
		Let $G$ be an arbitrary locally compact group. Every closed subset of $G^n$ is a set of spectral synthesis for $A^n(G)$ if and only if $G$ is discrete and $u\in \overline{uA^n(G)}\ \forall\ u\in A^n(G).$
	\end{thm}
	\begin{proof}
		Suppose that every closed subset of $G^n$ is a set of spectral synthesis for $A^n(G).$ Let $G_0$ denote the connected component of the identity in $G.$ By Theorem \ref{Inj_Thm_Sp_Syn}, every closed subset of $(G_0)^n$ is a set of spectral synthesis for $A^n(G_0).$ If $G_0$ is non-trivial, then it contains a proper normal subgroup $K$ such that $G_0/K$ is a non-trivial connected Lie group. Thus, by Lemma \ref{IPTConv}, every closed subset of $(G_0/K)^n$  is a set of spectral synthesis for $A^n(G_0/K).$ Now, it is well known that every nontrivial connected lie group contains a closed non-discrete abelian subgroup. Let $H$ denote that subgroup of $G_0/K.$ Again, by \ref{Inj_Thm_Sp_Syn}, every closed subset of $H$ is a set of spectral synthesis for $A^n(H)$ which contradicts Corollary \ref{MallThm2}. This forces us to conclude that $G_0=\{e\}$ and therefore, $G$ is totally connected.
		
		We now claim that $G$ is discrete. Suppose that $G$ admits an infinite compact subgroup $K.$ Then, by Corollary \ref{MallThm1}, $K^n$ contains a closed subset $E$ which is not a set of spectral synthesis for $A^n(K)$ and hence for $A^n(G)$ leading to a contradiction. Thus, we conclude that $G$ is discrete. Further, as every closed set is a set of spectral synthesis for $A^n(G),$ it is clear that $u\in\overline{uA^n(G)}$ for every $u\in A^n(G).$
		
		Converse is an easy consequence of Lemma \ref{clopen}.
	\end{proof}
	
	\section*{Acknowledgement}
	The second author would like to thank the Science and Engineering Board, India, for the core research grant with file no. CRG/2021/003087/MS.

\end{document}